\documentclass[12pt]{article}
\usepackage{amssymb}
\usepackage{amsthm}
\usepackage{amsmath}
\usepackage{graphicx}
\usepackage{enumerate}
\usepackage{pict2e}
\usepackage{framed}

\DeclareMathOperator{\Max}{Max}
\DeclareMathOperator{\Min}{Min}

\newtheorem{theorem}{Theorem}[section]
\newtheorem{definition}[theorem]{Definition}
\newtheorem{lemma}[theorem]{Lemma}
\newtheorem{proposition}[theorem]{Proposition}

\newtheorem{remark}[theorem]{Remark}
\newtheorem{example}[theorem]{Example}
\newtheorem{corollary}[theorem]{Corollary}
\title{Properties and special filters of pseudocomplemented posets}
\author{Ivan~Chajda and Helmut~L\"anger}
\date{}

\begin{document}

\footnotetext{Support of the research of the first author by IGA, project P\v rF~2023~010, and of the second author by the Austrian Science Fund (FWF), project I~4579-N, entitled ``The many facets of orthomodularity'', is gratefully acknowledged.}

\maketitle

\begin{abstract}
Investigating the structure of pseudocomplemented lattices started ninety years ago with papers by V.~Glivenko, G.~Birkhoff and O.~Frink and this structure was essentially developed by G.~Gr\"atzer. In recent years, some special filters in pseudocomplemented and Stone lattices have been studied by M.~Sambasiva~Rao. However, in some applications, in particular in non-classical logics with unsharp logical connectives, pseudocomplemented posets instead of lattices are used. This motivated us to develop an algebraic theory of pseudocomplemented posets, i.e.\ we derive identities and inequalities holding in such posets and we use them in order to characterize the so-called Stone posets. Then we adopt several concepts of special filters and we investigate their properties in pseudocomplemented posets. Moreover, we show how properties of these filters influence algebraic properties of the underlying pseudocomplemented posets.
\end{abstract}



\section{Introduction}

Apart from V.~Glivenko's early work \cite{Gl} from 1929, the study of pseudocomplemented lattices started in 1957 with a solution of problem~70 of the second edition \cite B of G.~Birkhoff's monograph on lattice theory which provided a characterization of Stone lattices by minimal prime ideals, see \cite{GS}. The idea of a triple was conceived by G~.Gr\"atzer in 1961, see e.g.\ the monograph \cite{Gr}, as a method for proving O.~Frink's conjecture on the representation of Stone lattices, see \cite F. More on pseudocomplemented and Stone lattices can be found in \cite{Sp} and \cite{Sz}.

Recently M.~Sambasiva~Rao \cite{SR} studied so-called median filters in pseudocomplemented distributive lattices. Other results on filters in pseudocomplemented posets can be found in \cite H.

In connection with non-classical logics the natural question arises if the theory of pseudocomplemented lattices can be extended to posets. (For another paper on pseudocomplemented posets cf.\ paper \cite H by R.~Hala\v s.) The authors already published papers treating this topic, see e.g.\ \cite{CHK}, \cite{CL21} and \cite{CL23}. A characterization of pseudocomplemented posets by triples similar to that for lattices mentioned above was settled for particular cases by the first author and R.~Hala\v s in \cite{CH}. For relative pseudocomplementation in posets, see also \cite{CL18} and \cite{CL20}. Moreover, a certain generalization of pseudocomplementation in posets can be found in \cite{CLP}.

The aim of this paper is two-fold.

At first, we collect identities known to be valid in pseudocomplemented and Stone lattices and establish their counterparts in pseudocomplemented and Stone posets. It is evident that not all of these identities can be transferred to posets, but a lot can be settled and proved. Section~3 contains an algebraic theory of pseudocomplemented and Stone posets which, to the authors' knowledge, was not published before.

The second part of the paper (Section~4) is devoted to several special filters in pseudocomplemented and Stone posets. We investigate relationships between these filters and connections between such filters and the structure of the underlying posets.

\section{Basic concepts}

In the following we identify singletons with their unique element, i.e.\ we will often write $x$ instead of $\{x\}$.

Let $(P,\leq)$ be a poset, $a,b\in P$ and $A,B\subseteq P$. We define
\begin{align*}
A\le B & :\Leftrightarrow x\le y\text{ for all }xin A\text{ and all }y\in B, \\
  L(A) & :=\{x\in P\mid x\leq A\}, \\
  U(A) & :=\{x\in P\mid A\leq x\}. 
\end{align*}
Instead of $L(\{a\})$, $L(\{a,b\})$, $L(A\cup\{b\})$, $L(A\cup B)$ and $L\big(U(A)\big)$ we simply write $L(a)$, $L(a,b)$, $L(A,b)$, $L(A,B)$ and $LU(A)$, respectively. Analogously we proceed in similar cases. The expression $a\vee b$ means that the supremum of $a$ and $b$ exists and that it is equal to $a\vee b$. Dually, $a\wedge b$ means that the infimum of $a$ and $b$ exists and that it is equal to $a\wedge b$. A {\em poset} $(P,\leq,0)$ with bottom element $0$ is called {\em pseudocomplemented} if for every $x\in P$ there exists a greatest element $x^*$ of $P$ satisfying $L(x,x^*)=0$ which is equivalent to $x\wedge x^*=0$. This element $x^*$ is called the {\em pseudocomplement} of $x$. We then write $(P,\leq,{}^*,0)$. It is easy to see that $1:=0^*$ is the top element of $(P,\leq)$. Hence any pseudocomplemented poset is bounded and we write $(P,\leq,{}^*,0,1)$. If $(P,\leq,{}^*,0,1)$ is a pseudocomplemented poset and $a,b\in P$ then it is elementary to check that the following holds:
\begin{enumerate}[(i)]
\item $a\leq b$ implies $b^*\leq a^*$,
\item $a\leq a^{**}$,
\item $a^{***}=a^*$,
\item $a\le b^*$ if and only if $L(a,b)=0$ if and only if $b\le a^*$.
\end{enumerate}
An element $a$ of $P$ is called {\em dense} if $a^*=0$. Let $D$ denote the set of all dense elements of $P$. Of course, $1\in D$.

\begin{lemma}\label{lem1}
	Let $(P,\le,{}^*,0,1)$ be a pseudocomplemented poset and $a,b\in P$. Then the following holds:
	\begin{enumerate}[{\rm(i)}]
		\item $U(a^*,b^*)=1$ implies $L(a,b)=0$,
		\item $\big(L(a,b)\big)^*\subseteq U(a^*,b^*)$,
		\item $\big(U(a,b)\big)^*\subseteq L(a^*,b^*)$,
		\item $U(a,a^*)\subseteq D$.
	\end{enumerate}
\end{lemma}

\begin{proof}
\
\begin{enumerate}[(i)]
	\item[(i)] If $c\in L(a,b)$ then $c^*\in U(a^*,b^*)=1$ and hence $c\le c^{**}=0$ which implies $c=0$.
	\item[(ii)] and (iii) follow from the fact that $^*$ is antitone.
	\item[(iv)] If $b\in U(a,a^*)$ then $b^*\in L(a^*,a^{**})=0$.
\end{enumerate}
\end{proof}

We define
\begin{align*}
A\le_1B & :\Leftrightarrow\text{ for every }a\in A\text{ there exists some }b\in B\text{ with }a\le b, \\
A\le_2B & :\Leftrightarrow\text{ for every }b\in B\text{ there exists some }a\in A\text{ with }a\le b, \\
  A=_1B & :\Leftrightarrow A\le_1B\text{ and }B\le_1A, \\
  A=_2B & :\Leftrightarrow A\le_2B\text{ and }B\le_2A, \\
	A^* & :=\{a^*\mid a\in A\}.
\end{align*}
It is easy to see that
\begin{align*}
 A\le B & \text{ implies }B^*\le A^*, \\
A\le_1B & \text{ implies }B^*\le_2A^*, \\
A\le_2B & \text{ implies }B^*\le_1A^*, \\
  A=_1B & \text{ implies }A^*=_2B^*, \\
  A=_2B & \text{ implies }A^*=_1B^*.
\end{align*}

In the sequel $\Max A$ and $\Min A$ denotes the set of all maximal and minimal elements of $A$, respectively. Let $\mathbf P=(P,\leq)$ be a poset. $\mathbf P$ is said to satisfy the {\em Ascending Chain Condition} (ACC, for short) if it has no infinite ascending chains and it is said to satisfy the {\em Descending Chain Condition} (DCC, for short) if it has no infinite descending chains. In particular, every finite poset satisfies both the ACC and DCC. Note that if $\mathbf P$ satisfies the ACC then every non-empty subset $A$ of $P$ has at least one maximal element, i.e.\ the set $\Max A$ is non-empty, and if $\mathbf P$ satisfies the DCC then every non-empty subset $A$ of $P$ has at least one minimal element, i.e.\ the set $\Min A$ is non-empty. If $\mathbf P$ satisfies the ACC and $A\subseteq B$ then $A\le_1\Max B$. If $\mathbf P$ satisfies the DCC and $A\subseteq B$ then $\Min B\le_2A$. It is evident that the sets $\Max A$ and $\Min A$ are antichains in $\mathbf P$.

\begin{example}
	The poset visualized in Figure~1
		
\vspace*{-3mm}

\begin{center}
	\setlength{\unitlength}{7mm}
	\begin{picture}(4,8)
		\put(2,1){\circle*{.3}}
		\put(1,3){\circle*{.3}}
		\put(3,3){\circle*{.3}}
		\put(1,5){\circle*{.3}}
		\put(3,5){\circle*{.3}}
		\put(2,7){\circle*{.3}}
		\put(2,1){\line(-1,2)1}
		\put(2,1){\line(1,2)1}
		\put(1,3){\line(0,1)2}
		\put(1,3){\line(1,1)2}
		\put(3,3){\line(-1,1)2}
		\put(3,3){\line(0,1)2}
		\put(2,7){\line(-1,-2)1}
		\put(2,7){\line(1,-2)1}
		\put(1.85,.3){$0$}
		\put(.35,2.85){$a$}
		\put(3.4,2.85){$b$}
		\put(.35,4.85){$c$}
		\put(3.4,4.85){$d$}
		\put(1.85,7.4){$1$}
		\put(1.2,-.75){{\rm Fig.~1}}
		\put(-1.3,-1.75){Pseudocomplemented poset}
	\end{picture}
\end{center}

\vspace*{8mm}
	
is pseudocomplemented, but not a lattice. We have
\[
\begin{array}{l|rrrrrr}
	x      & 0 & a & b & c & d & 1 \\
	\hline
	x^*    & 1 & b & a & 0 & 0 & 0 \\
	\hline
	x^{**} & 0 & a & b & 1 & 1 & 1
\end{array}
\]
and $D=\{c,d,1\}$.
\end{example}

\begin{example}
	The poset depicted in Figure~2
	
	\vspace*{-3mm}
	
	\begin{center}
		\setlength{\unitlength}{7mm}
		\begin{picture}(8,8)
			\put(4,1){\circle*{.3}}
			\put(3,3){\circle*{.3}}
			\put(5,3){\circle*{.3}}
			\put(1,5){\circle*{.3}}
			\put(3,5){\circle*{.3}}
			\put(5,5){\circle*{.3}}
			\put(7,5){\circle*{.3}}
			\put(4,7){\circle*{.3}}
			\put(4,1){\line(-1,2)1}
			\put(4,1){\line(1,2)1}
			\put(3,3){\line(-1,1)2}
			\put(3,3){\line(0,1)2}
			\put(3,3){\line(1,1)2}
			\put(5,3){\line(-1,1)2}
			\put(5,3){\line(0,1)2}
			\put(5,3){\line(1,1)2}
			\put(4,7){\line(-3,-2)3}
			\put(4,7){\line(-1,-2)1}
			\put(4,7){\line(1,-2)1}
			\put(4,7){\line(3,-2)3}
			\put(3.85,.3){$0$}
			\put(2.35,2.85){$a$}
			\put(5.4,2.85){$b$}
			\put(.35,4.85){$c$}
			\put(2.35,4.85){$d$}
			\put(5.4,4.85){$e$}
			\put(7.4,4.85){$f$}
			\put(3.85,7.4){$1$}
			\put(3.2,-.75){{\rm Fig.~2}}
			\put(.7,-1.75){Pseudocomplemented poset}
		\end{picture}
	\end{center}
	
	\vspace*{8mm}
	
	is pseudocomplemented, but not a lattice. We have
	\[
	\begin{array}{l|rrrrrrrr}
		x      & 0 & a & b & c & d & e & f & 1 \\
		\hline
		x^*    & 1 & f & c & f & 0 & 0 & c & 0 \\
		\hline
		x^{**} & 0 & c & f & c & 1 & 1 & f & 1
	\end{array}
	\]
	and $D=\{d,e,1\}$.
\end{example}

\begin{remark}
	As mentioned above, the sets $\Max L(x,y)$ and $\Min U(x,y)$ are antichains for all $x,y\in P$. However, if a subset $A$ is an antichain then $A^*$ need not be an antichain. For example, the set $\{c,d\}$ in the poset visualized in Fig.~2 is an antichain but $\{c,d\}^*=\{c^*,d^*\}=\{0,f\}$ is not an antichain. Even the set $\big(\Max L(x,y)\big)^*$ need not be an antichain, see the following example.
\end{remark}

\begin{example}
The poset depicted in Figure~3

\vspace*{-3mm}

\begin{center}
	\setlength{\unitlength}{7mm}
	\begin{picture}(6,10)
		\put(3,1){\circle*{.3}}
		\put(2,3){\circle*{.3}}
		\put(4,3){\circle*{.3}}
		\put(3,5){\circle*{.3}}
		\put(5,5){\circle*{.3}}
		\put(1,7){\circle*{.3}}
		\put(3,7){\circle*{.3}}
		\put(5,7){\circle*{.3}}
		\put(3,9){\circle*{.3}}
		\put(3,1){\line(-1,2)1}
		\put(3,1){\line(1,2)2}
		\put(2,3){\line(-1,4)1}
		\put(2,3){\line(1,2)1}
		\put(4,3){\line(-1,2)1}
		\put(3,5){\line(0,1)4}
		\put(3,5){\line(1,1)2}
		\put(5,5){\line(-1,1)2}
		\put(5,5){\line(0,1)2}
		\put(3,9){\line(-1,-1)2}
		\put(3,9){\line(1,-1)2}
		\put(2.85,.3){$0$}
		\put(1.35,2.85){$a$}
		\put(4.4,2.85){$b$}
		\put(2.35,4.85){$c$}
		\put(5.4,4.85){$d$}
		\put(.35,6.85){$e$}
		\put(2.35,6.85){$f$}
		\put(5.4,6.85){$g$}
		\put(2.85,9.4){$1$}
		\put(2.2,-.75){{\rm Fig.~3}}
		\put(-.3,-1.75){Pseudocomplemented poset}
	\end{picture}
\end{center}

\vspace*{8mm}

is pseudocomplemented, but not a lattice. We have
\[
\begin{array}{l|rrrrrrrrr}
	x      & 0 & a & b & c & d & e & f & g & 1 \\
	\hline
	x^*    & 1 & d & e & 0 & e & d & 0 & 0 & 0 \\
	\hline
	x^{**} & 0 & e & d & 1 & d & e & 1 & 1 & 1
\end{array}
\]
and $D=\{c,f,g,1\}$. We have
\[
\big(\Max L(f,g)\big)^*=(\{c,d\})^*=\{c^*,d^*\}=\{0,e\}
\]
which is not an antichain.
\end{example}

If $(P,\le)$ is a poset and $A$ a subset of $P$ and if $L(x^*,y)\neq0$ and $L(x,y^*)\neq0$ for all $x,y\in A$ with $x\neq y$ then both $A$ and $A^*$ are antichains, for instance if there would exist $a,b\in A$ with $a<b$ then $L(a,b^*)\subseteq L(b,b^*)=0$ and hence we would obtain $L(a,b^*)=0$, a contradiction. The other cases can be treated analogously.
	
\begin{remark}\label{rem1}
	Let $\mathbf P=(P,\le)$ be a poset, $A$ an antichain of $\mathbf P$ and $B\subseteq P$ and assume $A=_1B$. Then $A\subseteq B$. This can be seen as follows: Assume $a\in A$. Since $A\le_1B$ there exists some $b\in B$ with $a\le b$. Since $B\le_1A$ there exists some $c\in A$ with $b\le c$. Because of transitivity we have $a\le c$. Since $A$ is an antichain of $\mathbf P$ we conclude $a=c$ and hence $a=b\in B$. This shows $A\subseteq B$. In a similar way one can prove that $A\subseteq B$ provided $A=_2B$ instead of $A=_1B$. If both $A$ and $B$ are antichains of $\mathbf P$ and either $A=_1B$ or $A=_2B$, we obtain $A=B$. If $A=_2B$ then $A=1$ if and only if $B=1$. This can be seen as follows: Assume $A=1$. Then $B=_21$. Because of $B\le_21$ we have $B\ne\emptyset$. On the other hand, $1\le_2B$ implies $B=1$. The converse implication follows by symmetry.
\end{remark}

\section{Properties of pseudocomplements in bounded posets}

In this section we present identities and inequalities holding in pseudocomplemented posets and compare these identities and inequalities with the corresponding identities holding in pseudocomplemented lattices.

We start with a list of several important properties of bounded pseudocomplemented posets.

\begin{theorem}\label{th3}
	Let $\mathbf P=(P,\le,{}^*,0,1)$ be a pseudocomplemented poset. If $\mathbf P$ satisfies the {\rm ACC} then the following holds:
	\begin{enumerate}[{\rm(i)}]
		\item $\Max L(x^*,y^*)\approx_1\big(\Max L(x^*,y^*)\big)^{**}$,
		\item $\Max L\Big(x,\big(\Max L(x^*,y)\big)^*\Big)\approx x$,
		\item $\Max L(x,y)\leq_1\Max L(x^{**},y)\leq_1\Max L(x^{**},y^{**})$ for all $x,y\in P$,
		\item $\Max L(x,y)\leq_1\Max L(x,y^{**})\leq_1\Max L(x^{**},y^{**})$ for all $x,y\in P$,
		\item $\big(\Max L(x,y)\big)^{**}\leq_1\Max L(x^{**},y^{**})$ for all $x,y\in P$.
	\end{enumerate}
	If $\mathbf P$ satisfies the {\rm DCC} then the following holds:
	\begin{enumerate}
		\item[{\rm(vi)}] $\Min U(x,y)\leq_2\Min U(x^{**},y)\leq_2\Min U(x^{**},y^{**})$ for all $x,y\in P$,
		\item[{\rm(vii)}] $\Min U(x,y)\leq_2\Min U(x,y^{**})\leq_2\Min U(x^{**},y^{**})$ for all $x,y\in P$,
		\item[{\rm(viii)}] $\Min U(x^{**},y^{**})\le_2\big(\Min U(x,y)\big)^{**}$ for all $x,y\in P$.
	\end{enumerate}
	If $\mathbf P$ satisfies both the {\rm ACC} and the {\rm DCC} then the following holds:
	\begin{enumerate}
		\item[{\rm(ix)}] $\big(\Min U(x,y)\big)^*\approx_1\Max L(x^*,y^*)$,
		\item[{\rm(x)}] $\Min U(x^*,y^*)\le_2\big(\Max L(x,y)\big)^*$ for all $x,y\in P$.
	\end{enumerate}
\end{theorem}

\begin{proof}
	Let $a,b\in P$.
	\begin{enumerate}[(i)]
		\item $\Max L(a^*,b^*)\le a^*$ implies
		\[
		\big(\Max L(a^*,b^*)\big)^{**}\le a^{***}=a^*.
		\]
		Analogously, $\big(\Max L(a^*,b^*)\big)^{**}\le b^*$. Hence $\big(\Max L(a^*,b^*)\big)^{**}\subseteq L(a^*,b^*)$ whence $\big(\Max L(a^*,b^*)\big)^{**}\le_1\Max L(a^*,b^*)$. If $c\in\Max L(a^*,b^*)$ then $c\le c^{**}\in\big(\Max L(a^*,b^*)\big)^{**}$. This shows $\Max L(a^*,b^*)\le_1\big(\Max L(a^*,b^*)\big)^{**}$.
		\item We have $\Max L(a^*,b)\le a^*$ and hence $a\le a^{**}\le\big(\Max L(a^*,b)\big)^*$ whence
		\[
		L\Big(a,\big(\Max L(a^*,b)\big)^*\Big)=L(a).
		\]
		\item This follows from $L(a,b)\subseteq L(a^{**},b)\subseteq L(a^{**},b^{**})$.
		\item This follows from $L(a,b)\subseteq L(a,b^{**})\subseteq L(a^{**},b^{**})$.
		\item From (ii) and (i) we obtain $\big(\Max L(a,b)\big)^{**}\leq_1\big(\Max L(a^{**},b^{**})\big)^{**}=_1\Max L(a^{**},b^{**})$.
		\item This follows from $U(a^{**},b^{**})\subseteq U(a^{**},b)\subseteq U(a,b)$.
		\item This follows from $U(a^{**},b^{**})\subseteq U(a,b^{**})\subseteq U(a,b)$.
		\item We have $\Min U(a^{**},b^{**})\leq_2\big(\Max L(a^*,b^*)\big)^*=_2\big(\Min U(a,b)\big)^{**}$ according to (vii) and (vi).
		\item Since $a,b\le\Min U(a,b)$ we have $\big(\Min U(a,b)\big)^*\le a^*,b^*$ and hence $\big(\Min U(a,b)\big)^*\subseteq L(a^*,b^*)$ whence $\big(\Min U(a,b)\big)^*\le_1\Max L(a^*,b^*)$. Conversely, let $c\in\Max L(a^*,b^*)$. Then $c\le a^*,b^*$ and hence $a,b\le c^*$, i.e.\ $c^*\in U(a,b)$. Therefore there exists some $d\in\Min U(a,b)$ with $d\le c^*$. Now $c\le d^*\in\big(\Min U(a,b)\big)^*$ showing $\Max L(a^*,b^*)\le_1\big(\Min U(a,b)\big)^*$.
		\item Let $c\in\big(\Max L(a,b)\big)^*$. Then there exists some $d\in\Max L(a,b)$ with $d^*=c$. Because of $d\le a,b$ we have $a^*,b^*\le d^*$, i.e.\ $d^*\in U(a^*,b^*)$. Hence there exists some $e\in\Min U(a^*,b^*)$ with $e\le d^*$. Therefore $e\le c$.
	\end{enumerate}
\end{proof}

The following result is in fact a corollary of Theorem~\ref{th3}.

\begin{corollary}\label{cor1}
	Let $\mathbf P=(P,\le,{}^*,0,1)$ be a pseudocomplemented poset satisfying the {\rm ACC} and assume
	\begin{enumerate}
	\item[{\rm(1)}] $\big(\Max L(x,y)\big)^*\leq_2\big(\Max L(x^{**},y^{**})\big)^*$ for all $x,y\in P$.
	\end{enumerate}
	Then $\mathbf P$ satisfies the identities
	\begin{enumerate}[{\rm(i)}]
		\item $\big(\Max L(x,y)\big)^*\approx_2\big(\Max L(x^{**},y^{**})\big)^*$,
		\item $\big(\Max L(x,y)\big)^{**}\approx_1\Max L(x^{**},y^{**})$,
		\item $\big(\Max L(x,y)\big)^*\approx_2\big(\Max L(x^{**},y)\big)^*\approx_2\big(\Max L(x,y^{**})\big)^*$.
	\end{enumerate}
\end{corollary}

\begin{proof}
	\
	\begin{enumerate}[(i)]
		\item We have
		\[
		\big(\Max L(a,b)\big)^*\leq_2\big(\Max L(a^{**},b^{**})\big)^*\le_2\big(\Max L(a,b)\big)^*
		\]
		according to the assumption and to (iii) of Theorem~\ref{th3}.
		\item We have
		\[
		\big(\Max L(a,b)\big)^{**}=_1\big(\Max L(a^{**},b^{**})\big)^{**}=_1\Max L(a^{**},b^{**})
		\]
		according to (i) and to (i) of Theorem~\ref{th3}.
		\item We have
		\[
		L(a,b)\subseteq L(a^{**},b)\subseteq L(a^{**},b^{**})
		\]
		and hence
		\[
		\Max L(a,b)\le_1\Max L(a^{**},b)\le_1\Max L(a^{**},b^{**})
		\]
		whence
		\[
		\big(\Max L(a,b)\big)^*=_2\big(\Max L(a^{**},b^{**})\big)^*\le_2\big(\Max L(a^{**},b)\big)^*\le_2\big(\Max L(a,b)\big)^*
		\]
		according to (i). The assertion $\big(\Max L(a,b)\big)^*=_2\big(\Max L(a,b^{**})\big)^*$ follows by symmetry.
	\end{enumerate}
\end{proof}

The assumed inequality (1) is natural since it is satisfied in the case where the poset is a lattice (see Theorem~\ref{th1}). However, it is not satisfied in the poset from Fig.~2 since
\begin{align*}
\big(\Max L(d,e)\big)^* & =(\Max\{0,a,b\})^*=\{a,b\}^*=\{c,f\}\not\le_20=1^*= \\
                        & =(\Max\{0,a,b,c,d,e,f,1\})^*=\big(\Max L(1,1)\big)^*=\big(\Max L(d^{**},e^{**})\big)^*.
\end{align*}
In the sequel we will often deal with Stone posets. We define

\begin{definition}\label{def1}
	A {\em Stone poset} is a pseudocomplemented poset $(P,\le,{}^*,0,1)$ satisfying both the {\rm ACC} and the {\rm DCC} as well as
	\[
	\big(\Max L(x^*,y^*)\big)^*\le_2\Min U(x^{**},y^{**})\text{ for all }x,y\in P.
	\]
\end{definition}

Stone posets can be characterized by identities, see the following result.

\begin{theorem}\label{th5}
	Let $\mathbf P=(P,\le,{}^*,0,1)$ be a pseudocomplemented poset satisfying both the {\rm ACC} and the {\rm DCC}. Then the following are equivalent:
	\begin{enumerate}[{\rm(i)}]
		\item $\mathbf P$ is a Stone poset,
		\item $\mathbf P$ satisfies the identity $\big(\Max L(x^*,y^*)\big)^*\approx_2\Min U(x^{**},y^{**})$,
		\item $\mathbf P$ satisfies the identity $\big(\Min U(x,y)\big)^{**}\approx_2\Min U(x^{**},y^{**})$.
	\end{enumerate}
\end{theorem}

\begin{proof}
	$\text{}$ \\
	(i) $\Leftrightarrow$ (ii): \\
	This follows from (x) of Theorem~\ref{th3}. \\
	(ii) $\Leftrightarrow$ (iii): \\
	According to (ix) of Theorem~\ref{th3} we have $\big(\Min U(x,y)\big)^{**}\approx_2\big(\Max L(x^*,y^*)\big)^*$.
\end{proof}

Assertion (iii) of Theorem~\ref{th5} implies that $\mathbf P$ satisfies the identities
\begin{align*}
	\big(\Min U(x^*,y^*)\big)^{**} & \approx_2\Min U(x^{***},y^{***})\approx\Min U(x^*,y^*), \\
	\big(\Min U(x,y)\big)^* & \approx\big(\Min U(x,y)\big)^{***}\approx_1\big(\Min U(x^{**},y^{**})\big)^*.
\end{align*}

\begin{corollary}
	Let $\mathbf P=(P,\le,{}^*,0,1)$ be a pseudocomplemented poset satisfying the {\rm ACC} and the {\rm DCC} as well as the identity
	\begin{enumerate}
		\item[{\rm(2)}] $\big(\Min U(x^{**},y^{**})\big)^{**}\approx_2\Min U(x^{**},y^{**})$.
	\end{enumerate}
	Then $\mathbf P$ is a Stone poset.
\end{corollary}

\begin{proof}
	By (ix) of Theorem~\ref{th3}, $\mathbf P$ satisfies the identity
	\[
	\big(\Min U(x,y)\big)^*\approx_1\Max L(x^*,y^*)
	\]
	and hence also the identity
	\[
	\big(\Min U(x,y)\big)^{**}\approx_2\big(\Max L(x^*,y^*)\big)^*.
	\]
	Replacing $x$ by $x^{**}$ and $y$ by $y^{**}$ we obtain the identities
	\[
	\big(\Min U(x^{**},y^{**})\big)^{**}\approx_2\big(\Max L(x^{***},y^{***})\big)^*=_1\big(\Max L(x^*,y^*)\big)^*.
	\]
	Applying identity (2) we get the identity
	\[
	\Min U(x^{**},y^{**})\approx_2\big(\Max L(x^*,y^*)\big)^*
	\]
	which by (ii) of Theorem~\ref{th5} is equivalent to $\mathbf P$ being a Stone poset.
\end{proof}

Recall the following well-known concept. A {\em Stone lattice} is a distributive pseudocomplemented lattice $(L,\vee,\wedge,{}^*,0,1)$ satisfying the identity $x^*\vee x^{**}\approx1$.

In order to justify the name Stone posets for the posets defined in Definition~\ref{def1} we prove that such posets satisfy the identity $U(x^*,x^{**})\approx1$ which is equivalent to $x^*\vee x^{**}\approx1$. Hence such posets satisfy the same identity as Stone lattices.

\begin{lemma}\label{lem4}
	Let $\mathbf P=(P,\le,{}^*,0,1)$ be a Stone poset. Then $\mathbf P$ satisfies the identity
	\begin{enumerate}
	\item[{\rm(3)}] $U(x^*,x^{**})\approx1$.
	\end{enumerate}
\end{lemma}

\begin{proof}
According to Theorem~\ref{th5} we have
\[
\Min U(x^*,x^{**})\approx\Min U(x^{***},x^{**})\approx_2\big(\Max L(x^{**},x^*\big)^*\approx(\Max0)^*\approx0^*\approx1.
\]
Since both $\Min U(x^*,x^{**})$ and $\{1\}$ are antichains we obtain $\Min U(x^*,x^{**})\approx1$, i.e.\ $U(x^*,x^{**})\approx1$.
\end{proof}

Identity (3) is called the {\em Stone identity}.

\begin{remark}
	Let us recall that a distributive lattice is a Stone lattice if and only if it satisfies the identity
	\begin{enumerate}
		\item[{\rm(4)}] $(x^*\wedge y^*)^*\approx x^{**}\vee y^{**}$
	\end{enumerate}
	{\rm(}see Theorem~\ref{th2}{\rm)}. Since a pseudocomplemented poset satisfying both the {\rm ACC} and the {\rm DCC} is a Stone poset if and only if it satisfies the identity
	\[
	\big(\Max L(x^*,y^*)\big)^*\approx_2\Min U(x^{**},y^{**}),
	\]
	which is nothing else than identity {\rm(4)} translated to posets, we need not assume distributivity within the definition of a Stone poset.
\end{remark}

\begin{example}
	The posets depicted in Figure~4
	
	\vspace*{-3mm}
	
	\begin{center}
		\setlength{\unitlength}{7mm}
		\begin{picture}(4,10)
			\put(2,1){\circle*{.3}}
			\put(2,3){\circle*{.3}}
			\put(1,5){\circle*{.3}}
			\put(3,5){\circle*{.3}}
			\put(1,7){\circle*{.3}}
			\put(3,7){\circle*{.3}}
			\put(2,9){\circle*{.3}}
			\put(2,3){\line(0,-1)2}
			\put(2,3){\line(-1,2)1}
			\put(2,3){\line(1,2)1}
			\put(1,5){\line(0,1)2}
			\put(1,5){\line(1,1)2}
			\put(3,5){\line(-1,1)2}
			\put(3,5){\line(0,1)2}
			\put(2,9){\line(-1,-2)1}
			\put(2,9){\line(1,-2)1}
			\put(1.85,.3){$0$}
			\put(2.4,2.85){$a$}
			\put(.35,4.85){$b$}
			\put(3.4,4.85){$c$}
			\put(.35,6.85){$d$}
			\put(3.4,6.85){$e$}
			\put(1.85,9.4){$1$}
			\put(1.65,-.75){{\rm(a)}}
			\put(5.3,-1){{\rm Fig.~4}}
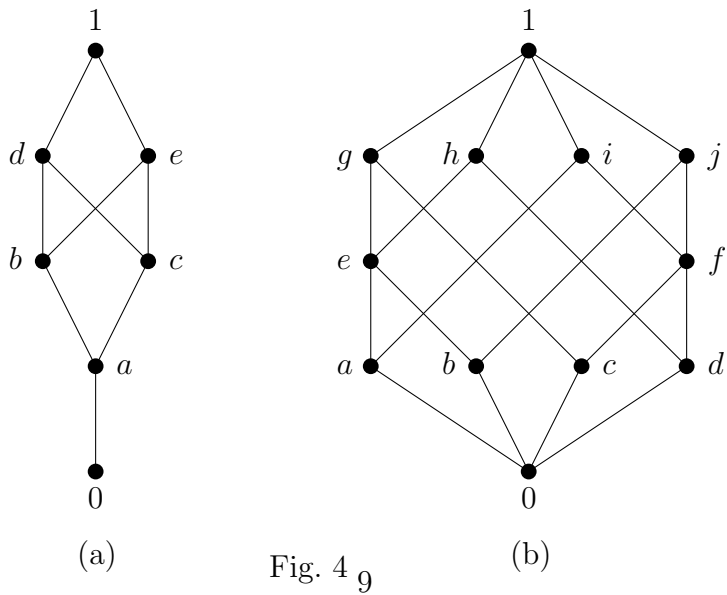
			\put(3,-2){Non-lattice Stone posets}
		\end{picture}
		\quad\quad\quad
		\setlength{\unitlength}{7mm}
		\begin{picture}(8,10)
			\put(4,1){\circle*{.3}}
			\put(1,3){\circle*{.3}}
			\put(3,3){\circle*{.3}}
			\put(5,3){\circle*{.3}}
			\put(7,3){\circle*{.3}}
			\put(1,5){\circle*{.3}}
			\put(7,5){\circle*{.3}}
			\put(1,7){\circle*{.3}}
			\put(3,7){\circle*{.3}}
			\put(5,7){\circle*{.3}}
			\put(7,7){\circle*{.3}}
			\put(4,9){\circle*{.3}}
			\put(4,1){\line(-3,2)3}
			\put(4,1){\line(-1,2)1}
			\put(4,1){\line(1,2)1}
			\put(4,1){\line(3,2)3}
			\put(1,3){\line(0,1)4}
			\put(1,3){\line(1,1)4}
			\put(3,3){\line(-1,1)2}
			\put(3,3){\line(1,1)4}
			\put(5,3){\line(-1,1)4}
			\put(5,3){\line(1,1)2}
			\put(7,3){\line(-1,1)4}
			\put(7,3){\line(0,1)4}
			\put(3,7){\line(-1,-1)2}
			\put(5,7){\line(1,-1)2}
			\put(4,9){\line(-3,-2)3}
			\put(4,9){\line(-1,-2)1}
			\put(4,9){\line(1,-2)1}
			\put(4,9){\line(3,-2)3}
			\put(3.85,.3){$0$}
			\put(.35,2.85){$a$}
			\put(2.35,2.85){$b$}
			\put(5.4,2.85){$c$}
			\put(7.4,2.85){$d$}
			\put(.35,4.85){$e$}
			\put(7.4,4.85){$f$}
			\put(.35,6.85){$g$}
			\put(2.35,6.85){$h$}
			\put(5.4,6.85){$i$}
			\put(7.4,6.85){$j$}
			\put(3.85,9.4){$1$}
			\put(3.65,-.75){{\rm(b)}}
		\end{picture}
	\end{center}
	
	\vspace*{17mm}
	
	are non-lattice Stone posets. For the poset in Fig.~4 {\rm(a)} we have
	\[
	\begin{array}{l|rrrrrrr}
		x      & 0 & a & b & c & d & e & 1 \\
		\hline
		x^*    & 1 & 0 & 0 & 0 & 0 & 0 & 0 \\
		\hline
		x^{**} & 0 & 1 & 1 & 1 & 1 & 1 & 1
	\end{array}
	\]
	and hence $D=\{a,b,c,d,e,1\}$. For the poset in Fig.~4 {\rm(b)} we have
	\[
	\begin{array}{l|rrrrrrrrrrrr}
		x      & 0 & a & b & c & d & e & f & g & h & i & j & 1 \\
		\hline
		x^*    & 1 & j & i & h & g & f & e & d & c & b & a & 0 \\
		\hline
		x^{**} & 0 & a & b & c & d & e & f & g & h & i & j & 1
	\end{array}
	\]
	and therefore it satisfies the identity $x^{**}\approx x$ and $D=\{1\}$. Moreover, both posets satisfy inequality {\rm(1)}.
\end{example}

\begin{remark}
	Recall that a {\em poset} is called {\em distributive} if it satisfies the identity
\[
L\big(U(x,y),z\big)\approx LU\big(L(x,z),L(y,z)\big).
\]
	The Stone identity and distributivity are independent identities. Namely, the poset from Fig.~1 is distributive but $U(a^*,a^{**})=U(a,b)=\{c,d,1\}\ne1$. On the other hand, the poset in Fig.~2 satisfies the Stone identity but it is not distributive since
	\begin{align*}
		L\big(U(c,d),e\big) & =L(e,1)=L(e)\ne\{0,a,b\}=L(d,e,1)=LU(0,a,b)= \\
		& =LU\big(L(c,e),L(d,e)\big).
	\end{align*}
\end{remark}

One can compare our previous results with well-known properties of pseudocomplemented lattices as found in \cite F -- \cite{Gr}, see the following two theorems.

\begin{theorem}\label{th1}
Every pseudocomplemented lattice satisfies the following identities:
\begin{itemize}
\item $(x\wedge y)^*\approx(x^{**}\wedge y)^*\approx(x\wedge y^{**})^*\approx(x^{**}\wedge y^{**})^*$,
\item $(x\wedge y)^{**}\approx x^{**}\wedge y^{**}$,
\item $x\wedge(x^*\wedge y)^*\approx x$,
\item $(x\vee y)^*\approx x^*\wedge y^*$.
\end{itemize}
\end{theorem}

The next result is well-known, see e.g.\ \cite{Gr}.

\begin{theorem}\label{th2}
If $\mathbf L=(L,\vee,\wedge,{}^*,0,1)$ is a distributive pseudocomplemented lattice then the following are equivalent:
\begin{enumerate}[{\rm(i)}]
\item $\mathbf L$ is a Stone lattice,
\item $\mathbf L$ satisfies the identity $(x^*\wedge y^*)^*\approx x^{**}\vee y^{**}$,
\item $\mathbf L$ satisfies the identity $(x\vee y)^{**}\approx x^{**}\vee y^{**}$.
\end{enumerate}
\end{theorem}

Our assumption in Definition~\ref{def1}, namely $\big(\Max L(x^*,y^*)\big)^*\le_2\Min U(x^{**},y^{**})$ for all $x,y\in P$, is quite natural since if the poset in question is a lattice, then it gets $(x^*\wedge y^*)^*\le x^{**}\vee y^{**}$ for all $x,y\in P$ which is automatically satisfied in every Stone lattice as follows from (ii) of Theorem~\ref{th2}.

\begin{remark}
Assertion {\rm(iii)} of Theorem~\ref{th2} implies that $\mathbf L$ satisfies the identities
\begin{align*}
(x^*\vee y^*)^{**} & \approx x^{***}\vee y^{***}\approx x^*\vee y^*, \\
       (x\vee y)^* & \approx(x\vee y)^{***}\approx(x^{**}\vee y^{**})^*
\end{align*}
{\rm(}see e.g.\ {\rm\cite{Gr})}
\end{remark}

\section{Filters}

The concept of a filter of a poset was already introduced in \cite{CH}. Recall that a non-empty subset of a poset $\mathbf P=(P,\le)$ is called a {\em filter} of $\mathbf P$ if
\begin{itemize}
	\item $a\in F$, $b\in P$ and $a\le b$ together imply $b\in F$,
	\item $a,b\in F$ implies $L(a,b)\cap F\ne\emptyset$.
\end{itemize}
The {\em filter} $F$ is called {\em proper} if $F\neq P$. For every $x\in P$ the set $F_x:=\{y\in P\mid x\le y\}$ is a filter of $\mathbf P$, the so-called {\em principal filter} of $\mathbf P$ generated by $x$. In case $\mathbf P$ satisfies the DCC every filter of $\mathbf P$ is principal. This can be seen as follows. Assume $\mathbf P$ to satisfy the DCC and let $F$ be a filter of $\mathbf P$ and $a\in F$. Because of the DCC there exists some minimal element $b$ of $F$ with $b\le a$. Assume that $b$ is not the smallest element of $F$. Then there exists some $c\in F$ with $b\not\le c$. Since $F$ is a filter of $\mathbf P$ there exists some $d\in L(b,c)\cap F$. Now $d=b$ would imply $b\le c$, a contradiction. This shows $d<b$ contradicting the minimality of $b$. In particular, every filter of a finite pseudocomplemented poset as well as of a Stone poset is principal.

\begin{example}
	All filters of the poset from Fig.~1 are given by:
	\begin{align*}
		F_0 & =\{0,a,b,c,d,1\}, \\
		F_a & =\{a,c,d,1\}, \\
		F_b & =\{b,c,d,1\}, \\
		F_c & =\{c,1\}, \\
		F_d & =\{d,1\}, \\
		F_1 & =\{1\}.
	\end{align*}
\end{example}

Hence the poset of all filters of a poset $\mathbf P$ satisfying the DCC ordered by set inclusion is dually isomorphic to $\mathbf P$ and hence it is lattice-ordered if and only if $\mathbf P$ has this property.

Let us note that e.g.\ the subset $D=\{c,d,1\}$ of $P$ from Fig.~1 is not a filter of $\mathbf P$ since $c,d\in D$, but $L(c,d)\cap D=\emptyset$.

\begin{lemma}\label{lem2}
	Let $\mathbf P=(P,\le,{}^*,0,1)$ be a pseudocomplemented poset, $F$ a proper filter of $\mathbf P$ and $a\in P$. Then the following holds:
	\begin{enumerate}[{\rm(i)}]
		\item $\{a,a^*\}\not\subseteq F$,
		\item if $a^*\in F$ and $b\in\overline a$ then $\overline b\not\subseteq F$.
	\end{enumerate}
\end{lemma}

\begin{proof}
\
\begin{enumerate}[(i)]
	\item The inclusion $\{a,a^*\}\subseteq F$ would imply $\{0\}\cap F=L(a,a^*)\cap F\ne\emptyset$ and hence $0\in F$, i.e.\ $F=P$ contradicting the fact that $F$ is proper.
	\item Assume $a^*\in F$ and $b\in\overline a$. Since $F$ is proper, we conclude $a\not\in F$ by (i). Now $b\in\overline a$ implies $a\in\overline b$. Together we have $a\in\overline b\setminus F$ showing $\overline b\not\subseteq F$.
\end{enumerate}
\end{proof}

Next we generalize two concepts introduced by M.~Sambasiva~Rao in \cite{SR} to pseudocomplemented posets.

Let $(P,\le,{}^*,0,1)$ be a pseudocomplemented poset and $A,B\subseteq P$. We define
\begin{align*}
	\overline A & :=\{x\in P\mid U(x^{**},y^{**})=1\text{ for all }y\in A\}, \\
	        A^D & :=\{x\in P\mid U(x,y)\subseteq D\text{ for all }y\in A\}.
\end{align*}

\begin{lemma}
	Let $(P,\le,{}^*,0,1)$ be a pseudocomplemented poset and $A\subseteq P$. Then $\overline{\emptyset}=\emptyset^D=P$ and $\overline P=P^D=D$.
\end{lemma}

\begin{proof}
	Let $a\in P$. Then the following are equivalent: $a\in\overline P$; $U(a^{**},x^{**})=1$ for all $x\in P$; $U(a^{**},0^{**})=1$; $U(a^{**},0)=1$; $U(a^{**})=1$; $a^{**}=1$; $a^*=0$; $a\in D$. This proves $\overline P=D$. Moreover, the following are equivalent: $a\in P^D$; $U(a,x)\subseteq D$ for all $x\in P$; $U(a,0)\subseteq D$; $U(a)\subseteq D$; $a\in D$. This shows $P^D=D$. The rest of the proof is trivial.
\end{proof}

\begin{lemma}\label{lem3}
	Let $\mathbf P=(P,\le,{}^*,0,1)$ be a pseudocomplemented poset, $a\in P$ and $A\subseteq P$. Then the following holds:
	\begin{enumerate}[{\rm(i)}]
		\item $\overline A\subseteq A^D$,
		\item if $\mathbf P$ is a Stone poset then $\overline A=A^D$ and hence $a^*\in a^D$.
	\end{enumerate}
\end{lemma}

\begin{proof}
Observe that $\overline A=\bigcap\limits_{x\in A}\overline x$ and $A^D=\bigcap\limits_{x\in A}x^D$.
\begin{enumerate}[(i)]
	\item Let $b\in\overline a$. Then $U(a^{**},b^{**})=1$. Let $c\in U(a,b)$. Then $c^{**}\in U(a^{**},b^{**})=1$ and hence $c^*=0$, i.e.\ $c\in D$. This shows $U(a,b)\subseteq D$ which means $b\in a^D$. We have proved $\overline a\subseteq a^D$.
	\item Assume $\mathbf P$ to be a Stone poset. According to (i), $\overline a\subseteq a^D$. Now let $b\in a^D$. Then $U(a,b)\subseteq D$ and hence $\Min U(a,b)\subseteq D$ whence $\big(\Min U(a,b)\big)^*=0$, i.e.\ $\big(\Min U(a,b)\big)^{**} =1$. Due to Theorem~\ref{th5} (iii) and Remark~\ref{rem1} the last statement is equivalent to $\Min U(a^{**},b^{**})=1$ which means the same as $U(a^{**},b^{**})=1$, i.e.\ $b\in\overline a$. This shows $a^D\subseteq\overline a$ and hence $\overline a=a^D$. According to Lemma~\ref{lem4} we have $U(a^{***},a^{**})=U(a^*,a^{**})=1$ and hence $a^*\in\overline a\subseteq a^D$ by (i).
\end{enumerate}
\end{proof}

\begin{remark}
	The pair $(A\mapsto\overline A,A\mapsto\overline A)$ of mappings is the Galois-correspondence between $(2^P,\subseteq)$ and $(2^P,\subseteq)$ induced by the binary relation $\{(x,y)\in P^2\mid U(x^{**},y^{**})=1\}$ on $P$. Hence $A\mapsto\overline{\overline A}$ is a closure operator on $(2^P,\subseteq)$ and therefore the set $\{\overline A\mid A\subseteq P\}$ of all subsets of $P$ closed under this closure operator forms a closure system of $P$, i.e.\ it is closed under arbitrary intersections. This means that the closed subsets of $P$ form a complete lattice with respect to set inclusion with smallest element $\overline P=D$ and greatest element $\overline\emptyset=P$ and we have the following properties for $A,B\subseteq P$:
	\begin{align*}
		A & \subseteq B\text{ implies }\overline B\subseteq\overline A, \\
		A & \subseteq\overline{\overline A}, \\
		\overline{\overline{\overline A}} & =\overline A, \\
		A & \subseteq\overline B\text{ if and only if }B\subseteq\overline A.
	\end{align*}
	Analogously, the pair $(A\mapsto A^D,A\mapsto A^D)$ is the Galois-correspondence between $(2^P,\subseteq)$ and $(2^P,\subseteq)$ induced by the binary relation $\{(x,y)\in P^2\mid U(x,y)\subseteq D\}$ on $P$. Hence the set $\{A^D\mid A\subseteq P\}$ of all subsets of $P$ closed under the closure operator $A\mapsto A^{DD}$ forms a complete lattice with respect to set inclusion and we have analogous results as before.
\end{remark}

Note that $\overline A$ as well as $A^D$ are order filters in $(P,\le)$ and that for $a\in P$ we have $\overline{F_a}=\overline a$ as well as $(F_a)^D=a^D$.

The following result shows how the property of being a Stone poset influences the properties of filters and, conversely, which properties of filters imply the Stone identity.

\begin{theorem}
	Let $\mathbf P=(P,\le,{}^*,0,1)$ be a pseudocomplemented poset satisfying both the {\rm ACC} and the {\rm DCC}. Consider the following statements:
	\begin{enumerate}[{\rm(i)}]
		\item $\mathbf P$ is a Stone poset,
		\item $\overline F=F^D$ for every filter $F$ of $\mathbf P$,
		\item for each filters $F,G$ of $\mathbf P$ we have $F\cap G\subseteq D$ if and only if $F\subseteq\overline G$,
		\item $\overline{\overline x}=\overline{x^*}$ for all $x\in P$,
		\item $\overline{\overline x}\subseteq\overline{x^*}$ for all $x\in P$,
		\item $\mathbf P$ satisfies the Stone identity $U(x^*,x^{**})\approx1$.
	\end{enumerate}
	Then {\rm(i)} $\Rightarrow$ {\rm(ii)} $\Leftrightarrow$ {\rm(iii)} $\Rightarrow$ {\rm(iv)} $\Rightarrow$ {\rm(v)} $\Leftrightarrow$ {\rm(vi)}.
\end{theorem}

\begin{proof}
	Let $a,b,c\in P$. Since $\mathbf P$ satisfies the DCC, every filter of $\mathbf P$ is principal. Moreover, $\overline{F_a}=\overline a$ and $(F_a)^D=a^D$. \\
	(i) $\Rightarrow$ (ii): \\
	This follows from Lemma~\ref{lem3}. \\
	(ii) $\Leftrightarrow$ (iii) \\
	This follows from the fact that the statements
	\[
	F_a\cap F_b\subseteq D; U(a,b)\subseteq D; a\in b^D
	\]
	are equivalent and that the statements
	\[
	F_a\subseteq\overline{F_b}; F_a\subseteq\overline b; a\in\overline b
	\]
	are equivalent. \\
	(ii) $\Rightarrow$ (iv): \\
	Since $a,a^*\le U(a,a^*)$ we have $\big(U(a,a^*)\big)^*\le a^*,a^{**}$ and hence $\big(U(a,a^*)\big)^*=0$ which implies $U(a,a^*)\subseteq D$, i.e.\ $a^*\in a^D=\overline a$ by (iii) from which we conclude $\overline{\overline a}\subseteq\overline{a^*}$. Conversely, assume $b\in\overline{a^*}$ and $c\in\overline a$. Then $U(a^{**},c^{**})=1$ and hence $L(a^*,c^*)=0$ according to Lemma~\ref{lem1}. This shows $c^*\le a^{**}$. Since $b\in\overline{a^*}$ we have $U(a^{***},b^{**})=1$ and hence $L(a^{**},b^*)=0$ again by Lemma~\ref{lem1}. Because of $c^*\le a^{**}$ we obtain $L(b^*,c^*)=0$. According to Lemma~\ref{lem1} we conclude $\big(U(b,c)\big)^*\subseteq L(b^*,c^*)=0$ and hence $\big(U(b,c)\big)^*=0$, i.e.\ $U(b,c)\subseteq D$ which means $b\in c^D=\overline c$ by (iii). Thus we have proved $b\in\overline x$ for all $x\in\overline a$ showing $b\in\overline{\overline a}$. Therefore $\overline{a^*}\subseteq\overline{\overline a}$. Together we obtain $\overline{\overline a}=\overline{a^*}$. \\
	(iv) $\Rightarrow$ (v): \\
	This is clear. \\
	(v) $\Leftrightarrow$ (vi): \\
	The following are equivalent: $U(a^*,a^{**})=1$; $a\in\overline{a^*}$; $\overline{\overline a}\subseteq\overline{a^*}$.
\end{proof}

Now we generalize a further concept (concerning filters) from \cite{SR} to pseudocomplemented posets as follows.

\begin{definition}
	We call a {\em filter} $F$ of a pseudocomplemented poset $(P,\le,{}^*,0,1)$ {\em coherent} if $x\in F$, $y\in P$ and $\overline x=\overline y$ together imply $y\in F$.
\end{definition}

\begin{lemma}
	Let $\mathbf P=(P,\le,{}^*,0,1)$ be a Stone poset and $F$ a filter of $\mathbf P$ and assume that for every $x\in P$, $x^{**}\in F$ implies $x\in F$. Then $F$ is coherent.
\end{lemma}

\begin{proof}
	Let $a,b,c\in P$. Since $\mathbf P$ satisfies the DCC, $F$ is principal as pointed out at the beginning of Section~4. Assume $b\in F_a$ and $\overline b=\overline c$. Since $\mathbf P$ is a Stone poset we have $U(a^*,a^{**})=1$ and hence $U(a^{***},b^{**})=U(a^*,b^{**})=1$ which shows $a^*\in\overline b=\overline c$. Therefore $U(a^*,c^{**})=U(a^{***},c^{**})=1$ whence $L(a,c^*)=0$ according to Lemma~\ref{lem1} showing $a\le c^{**}$, i.e.\ $c^{**}\in F_a$ which implies $c\in F_a$ according to our assumption.
\end{proof}

For every pseudocomplemented poset $(P,\le,{}^*,0,1)$ we define an operator $\pi$ as follows:
\[
\pi(A):=\{x\in P\mid\text{for all }y\in P\text{ there exists some }(z,u)\in\overline x\times A\text{ with }y\in\Min U(z,u)\}
\]
for all $A\subseteq P$. We use this operator for introducing the next concept.

\begin{definition}
	Let $\mathbf P=(P,\le,{}^*,0,1)$ be a pseudocomplemented poset and $F$ a filter of $\mathbf P$. Then $F$ is called
\begin{itemize}
	\item {\em strongly coherent} if $\pi(F)=F$,
	\item {\em closed} if $\overline{\overline F}=F$.
\end{itemize}
\end{definition}

\begin{lemma}
	Let $\mathbf P=(P,\le,{}^*,0,1)$ be a pseudocomplemented poset and $F$ a filter of $\mathbf P$ and assume $F$ to be either strongly coherent or closed. Then $F$ is coherent.
\end{lemma}

\begin{proof}
	Let $a\in F$ and $b\in P$ and suppose $\overline a=\overline b$. First assume $F$ to be strongly coherent. Then $a\in\pi(F)$. This means that for every $c\in P$ there exists some $(d,e)\in\overline a\times F$ with $c\in\Min U(d,e)$. Since $\overline a=\overline b$ we have that for every $c\in P$ there exists some $(d,e)\in\overline b\times F$ with $c\in\Min U(d,e)$. This shows $b\in\pi(F)=F$. Hence $F$ is coherent. If $F$ is closed then $b\in\overline{\overline b}=\overline{\overline a}\subseteq\overline{\overline F}=F$ proving $F$ to be coherent.
\end{proof}

The next types of filters are defined as follows.

\begin{definition}
	Let $\mathbf P=(P,\le,{}^*,0,1)$ be a pseudocomplemented poset and $F$ a filter of $\mathbf P$. Then $F$ is called
	\begin{itemize}
		\item {\em prime} if it is proper and if $x,y\in P$ and $U(x,y)\subseteq F$ together imply $\{x,y\}\cap F\ne\emptyset$,
		\item a {\em $D$-filter} if $D\subseteq F$,
		\item {\em maximal} if it is proper and if for every proper filter $G$ of $\mathbf P$, $F\subseteq G$ implies $F=G$.
	\end{itemize}
\end{definition}

\begin{theorem}\label{th4}
	Let $\mathbf P=(P,\le,{}^*,0,1)$ be a pseudocomplemented poset and $F$ a prime filter of $\mathbf P$. Consider the following statements:
	\begin{enumerate}[{\rm(i)}]
		\item $F$ is a $D$-filter,
		\item for any $x\in P$, $x\in F$ if and only if $x^*\not\in F$,
		\item for any $x\in P$, $x\in F$ if and only if $x^{**}\in F$,
		\item $x\in F$, $y\in P$ and $x^*=y^*$ together imply $y\in F$,
		\item $\overline x\subseteq F$ for all $x\in P\setminus F$.
	\end{enumerate}
	Then {\rm(i)} $\Leftrightarrow$ {\rm(ii)} $\Leftrightarrow$ {\rm(iii)} $\Leftrightarrow$ {\rm(iv)} $\Rightarrow$ {\rm(v)}.
\end{theorem}

\begin{proof}
	Let $a,b\in P$. \\
	(i) $\Rightarrow$ (ii): \\
	Because of Lemma~\ref{lem1} (iv) we have $U(a,a^*)\subseteq D\subseteq F$. Since $F$ is prime this implies $\{a,a^*\}\cap F\neq\emptyset$. On the other hand, $\{a,a^*\}\subseteq F$ is not possible because of Lemma~\ref{lem2}. \\
	(ii) $\Rightarrow$ (iii): \\
	If $a\in F$ then $a^{**}\in F$ because of $a\le a^{**}$. Conversely, if $a^{**}\in F$ then $a^*\notin F$ because of Lemma~\ref{lem2} which implies $a\in F$ according to (ii). \\
	(iii) $\Rightarrow$ (iv): \\
	If $a\in F$ and $a^*=b^*$ then $b^{**}=a^{**}\in F$ because of (iii) and hence $b\in F$ again according to (iii). \\
	(iv) $\Rightarrow$ (i): \\
	If $a\in D$ then $1\in F$ and $1^*=0=a^*$ from which we conclude $a\in F$ by (iv). \\
	(i) $\Rightarrow$ (v): \\
	Assume $a\in P\setminus F$ and $b\in\overline a$. Then $b\in a^D$ according to Lemma~\ref{lem3} whence $U(a,b)\subseteq D\subseteq F$. Since $a\not\in F$ and $F$ is prime, we conclude $b\in F$.
\end{proof}

We can use Theorem~\ref{th4} and Lemma~\ref{lem2} to prove the following result.

\begin{proposition}
	Let $\mathbf P=(P,\le,{}^*,0,1)$ be a pseudocomplemented poset and $F$ a filter of $\mathbf P$. Then any of the following statements implies the next one:
	\begin{enumerate}[{\rm(i)}]
		\item $F$ is a prime $D$-filter of $\mathbf P$,
		\item $x\in P\setminus F$ implies $x^*\in F$,
		\item $F$ is a maximal filter of $\mathbf P$.
	\end{enumerate}
\end{proposition}

\begin{proof}
	$\text{}$ \\
	(i) $\Rightarrow$ (ii): \\
	This follows from Theorem~\ref{th4}. \\
	(ii) $\Rightarrow$ (iii): \\
	Assume that (ii) holds and (iii) does not hold. Then there exists some proper filter $G$ of $\mathbf P$ strictly including $F$. Choose $a\in G\setminus F$. Since $a\in P\setminus F$ we conclude $a^*\in F$ by (ii). But then $a,a^*\in G$ contradicting the fact that $G$ is proper (see Lemma~\ref{lem2}).
\end{proof}

Hence every prime $D$-filter of a pseudocomplemented poset is maximal.

The last kind of filters generalized from Sambasiva Rao's paper \cite{SR} to pseudocomplemented posets is defined as follows.

\begin{definition}
	A filter $F$ of a pseudocomplemented poset $\mathbf P=(P,\le,{}^*,0,1)$ is called a {\em median filter} of $\mathbf P$ if $\overline x\not\subseteq F$ for all $x\in F$.
\end{definition}

We can characterize median filters as follows.

\begin{proposition}
	Let $\mathbf P=(P,\le,{}^*,0,1)$ be a pseudocomplemented poset and $F$ a prime $D$-filter of $\mathbf P$. Then the following are equivalent:
	\begin{enumerate}[{\rm(i)}]
		\item $F$ is median,
		\item for all $x\in P$, $x\in F$ if and only if $\overline x\not\subseteq F$,
		\item for all $x\in P$, $x^{**}\in F$ implies $\overline x\not\subseteq F$.
	\end{enumerate}
\end{proposition}

\begin{proof}
	Let $a\in P$. \\
	(i) $\Rightarrow$ (ii): \\
	If $a\in F$ then $\overline a\not\subseteq F$ because of (i), and if $a\not\in F$ then $\overline a\subseteq F$ according to Lemma~\ref{lem5}. \\
	(ii) $\Rightarrow$ (iii): \\
	Assume $a^{**}\in F$. Then $a\in F$ according to Theorem~\ref{th4}. By (ii) this implies $\overline a\not\subseteq F$. \\
	(iii) $\Rightarrow$ (i): \\
	Assume $a\in F$. Since $a\le a^{**}$ and $F$ is a filter of $\mathbf P$ we have $a^{**}\in F$. According to (iii) this implies $\overline a\not\subseteq F$.
\end{proof}

Finally, we show some properties of median $D$-filters.

\begin{theorem}
	Let $\mathbf P=(P,\le,{}^*,0,1)$ be a pseudocomplemented poset and $F$ a median $D$-filter of $\mathbf P$. Then the following holds:
	\begin{enumerate}[{\rm(i)}]
		\item If $F$ is prime then for all $x\in P$, $x\in F$ if and only if $\overline{\overline x}\subseteq F$,
		\item $F$ is coherent.
	\end{enumerate}
\end{theorem}

\begin{proof}
	Let $a\in F$. Since $F$ is median there exists some $c\in\overline a\setminus F$.
	\begin{enumerate}[(i)]
		\item Assume $F$ to be prime. Let $b\in\overline{\overline a}$. Then $c\in\overline a\subseteq\overline b$. Since $c\in P\setminus F$ we have $\overline c\subseteq F$ according to Theorem~\ref{th4}. Hence $b\in\overline{\overline b}\subseteq\overline c\subseteq F$. This shows $\overline{\overline a}\subseteq F$. That for $x\in P$, $\overline{\overline x}\subseteq F$ implies $x\in F$ follows from $x\in\overline{\overline x}$.
		\item Let $b\in P$ and assume $\overline a=\overline b$. Then $c\in\overline a=\overline b\subseteq b^D$ according to Lemma~\ref{lem3}. This implies $U(b,c)\subseteq D\subseteq F$. Since $F$ is prime and $c\not\in F$ we conclude $b\in F$.
	\end{enumerate}
\end{proof}

Authors' addresses:

Ivan Chajda \\
Palack\'y University Olomouc \\
Faculty of Science \\
Department of Algebra and Geometry \\
17.\ listopadu 12 \\
771 46 Olomouc \\
Czech Republic \\
ivan.chajda@upol.cz

Helmut L\"anger \\
TU Wien \\
Faculty of Mathematics and Geoinformation \\
Institute of Discrete Mathematics and Geometry \\
Wiedner Hauptstra\ss e 8-10 \\
1040 Vienna \\
Austria, and \\
Palack\'y University Olomouc \\
Faculty of Science \\
Department of Algebra and Geometry \\
17.\ listopadu 12 \\
771 46 Olomouc \\
Czech Republic \\
helmut.laenger@tuwien.ac.at

\end{document}